\documentclass[a4paper,11pt,pdftex]{article}
\usepackage{a4wide}
\usepackage{amsmath,amssymb,amsfonts,amsthm}
\usepackage{graphicx}
\usepackage{pb-diagram}
\usepackage{color}
\usepackage{calrsfs}
\usepackage[margin=10pt,font=small]{caption}

\newcommand{\bd}{\partial X}
\newcommand{\bdf}{\partial F_{g}}
\newcommand{\ct}{E_{X}}

\newcommand{\Tr}{\text{Tr}}
\newcommand{\ZX}{\zeta[X]}
\newcommand{\ZXo}{\zeta[X_{1}]}
\newcommand{\ZXt}{\zeta[X_{2}]}

\newcommand{\nn}{\nonumber}

\title{Graphs, spectral triples and Dirac zeta functions}
\author{Jan Willem de Jong}

\begin{document}
\include{def}
\maketitle

\abstract
To a finite, connected, unoriented graph of Betti-number $g \geq 2$ and valencies $\geq 3$ we associate a finitely summable, commutative spectral triple (in the sense of Connes), whose induced zeta functions encode the graph. This gives another example where non-commutative geometry provides a rigid framework for classification.

\section{Introduction}
Let $X$ be a finite, connected, unoriented graph with first Betti number $g \geq 2$ (from now on called the genus) and valencies $\geq 3$. Let $\ct$ denote the universal covering tree of $X$. The fundamental group $\Gamma$ of $X$, which is a free group of rank $g$, acts on $\ct$ and we have
\bea
X \cong \ct / \Gamma.
\eea
The action of the fundamental group induces an action on the boundary $\bd$ of $\ct$. The dynamics of the action on the boundary endowed with the Patterson-Sullivan measure encodes the graph, that is, equivariant, non-singular homeomorphisms only exist between boundaries coming from isomorphic graphs (see Proposition \ref{rigid}). In this article we will show that, using this, the graph can be encoded by a finitely summable, commutative spectral triple (see Definition \ref{s3}). This is a notion from non-commutative geometry describing a non-commutative analogue of Riemannian manifolds (see \cite{con} for an extensive treatment of non-commutative geometry). In fact, we will show that the zeta function formalism available for finitely summable spectral triples encodes the Patterson-Sullivan measure and hence the isomorphism class of the graph.\\

This article is inspired by \cite{cor} where a similar construction was applied to Riemann surfaces, which was based on \cite{cor2} where, among other things, $\theta$-summable spectral triples are constructed for actions on trees, which was in fact a refinement of \cite{ara}. The chain of ideas started in \cite{chri} with a construction of spectral triples for AF $C^{*}$-algebras and closely related work in \cite{s3r}.\\

{\it Construction of the spectral triple $\mathcal{S}_{X} = (A,H,D)$.} Given a graph of genus $g$ as above, then after representing $\Gamma$ as group of isometries of $E_{X}$, there exists a $\Gamma$-equivariant homeomorphism $\Phi_{X}: \p F_{g} \to \bd$ (see Proposition \ref{giso}), where $\partial F_{g}$ is the boundary of the Cayley graph of $\Gamma$.
For the algebra we take $A=C(\p F_{g},\C)$, the continuous, complex valued functions on the boundary of the Cayley graph. Note that by the Gelfand-Naimark theorem this algebra encodes the topology on $\p F_{g}$. As a side remark, the natural candidate for $A$ in non-commutative geometry would have been 
$A \rtimes \Gamma$, but by Connes' result on non-amenable groups, this would not give rise to a \emph{finitely} summable spectral triple (see \cite{con}, Theorems 17 and 19, pp. 214-215).
We will also need $A_{\infty}=C(\p F_{g},\Z) \otimes_{\Z} \C$, the locally constant functions on the boundary. The Hilbert-space $H$ will be the completion of $A$ with respect to integration, i.e. $||f ||^{2} = \int |f|^{2} d\nu$, of the induced measure from the Patterson-Sullivan measure on $\bd$ to $\p F_{g}$ via $\Phi_{X}$. The Dirac operator $D$ is composed of projection operators depending on the word grading in $\Gamma$ (see Definition \ref{dirac}).\\
The result is the following (see Theorem \ref{ctrip}):

\begin{Theo}
The spectral triple $\mathcal{S}_{X}$ determines the graph $X$.
\end{Theo}

In fact, this theorem follows from {\it zeta function rigidity}, which means that the (spectral) zeta functions, defined by $\zeta_{a}^{X}(s)=\Tr(a|D|^{s})$ for $a \in A_{\infty}$, already contain all the information of the graph and hence of the triple (see Theorem \ref{czet}). The construction depends on the choice of an origin and on a minimal set of chosen generators for $\Gamma$. To deal with these choices, we collect them all in a set, called $\zeta[X]$ (see Definition \ref{equivzeta}).

\begin{Theo}[Main theorem]
Let $X_{1},X_{2}$ be finite, connected graphs of genus $g \geq 2$ and valencies $\geq 3$. Then either
\bea
\ZXo \cap \ZXt = \emptyset,
\eea
or
\bea
\ZXo = \ZXt \text{ and } X_{1} \cong X_{2} \text{ as graphs}.
\eea
Here the intersection is defined as follows:\\
One starts by comparing at the unit of the algebras, i.e. $\zeta_{1}^{X_{i}}$: If these are different, the intersection is defined to be $\emptyset$; If they are the same, the genus $g$ is the same by Proposition \ref{zetaone}. Now consider $\partial F_{g}$ and fix the algebra $A_{\infty}$. The infinitely long rows $(\zeta_{a}^{X_{i}})_{a \in A_{\infty}} \in \zeta[X_{i}]$ are now indexed by the same algebra $A_{\infty}$ and we can compare elements $(\zeta_{a}^{X_{i}})_{a \in A_{\infty}} \in \zeta[X_{i}]$.
\end{Theo}

\begin{Rem}
This theorem answers the non-commutative isospectrality question for graphs, namely that one can retrieve the graph by its non-commutative spectra. The spectrum of the Dirac operator itself, contained in $\zeta_{1}(s)=\Tr(|D|^{s})$, does not determine the graph. In fact, it only determines the genus by the innocent zeta function (cf. formula \ref{zeta1}):
\bea 
\zeta^{X}_{1}(s) = 1 + (2g)^{3s}(2g-1)\Bigl\{ \frac{1-(2g-1)^{3s-1}}{1-(2g-1)^{3s+1}} \Bigr\}
\eea
\end{Rem}

A further motivation for this article was to understand what morphisms in the (so far non-existing) category of spectral triples should be. To achieve this, one can for instance try to map the objects of known categories (in a sensible manner) into the objects of 'the category of spectral triples' (i.e. the spectral triples itself) and then study what happens with morphisms in the known category. In this article we map the objects of the category of graphs into the objects of the 'category of spectral triples'. Note that the constructed triple is in fact commutative. In the framework of spectral triples there should also be morphisms not induced by the commutative analogue, but should for instance also entail Morita equivalence.

\section{Preliminaries}
In this section we will recall some well-known facts about graphs.
\begin{Tex}
Let $X$ be a finite, connected, unoriented graph with first Betti number $g \geq 2$ and valencies $\geq 3$. We will assume this throughout this paper. Let $\ct$ denote the universal covering tree of $X$ and let $\Gamma$ be the fundamental group of $X$. Then $\Gamma$ is a free group of rank $g$ and acts freely on $\ct$, with
\bea
X \cong \ct / \Gamma
\eea
(see for instance the introduction of \cite{ro1}).
\end{Tex}

\begin{Rem}
We will not distinguish between regarding an element of $\Gamma$ as closed path or as isometry of $E_{X}$, it will be clear from the context.
\end{Rem}

\begin{Defi}[Distance on $\ct$] \label{distance}
The tree $\ct$ is a complete, geodesic, metric space for a natural distance function $l$: Every edge is defined to be a closed interval of fixed length $L \in \R_{>0}$. Since two points in the tree are connected by a unique sequence of edges, this induces a distance between them.
\end{Defi}

\begin{Defi}[Boundary]
Let $X$ be a finite, connected graph with universal covering tree $\ct$.\\
The boundary $\bd$ of $\ct$ consists of all infinite words $w=w_{0}w_{1}w_{2}\ldots$ in adjoining, unoriented edges, without backtracking, in the covering tree. Two such words are considered the same if they agree from some point on; more precisely $w=w'$ if there exist $N,K$ such that for all $n \geq N$ we have $w_{n+K}=w_{n}'$.
\end{Defi}

\begin{Rem}
One can also define a boundary by means of a 'scalar product' (see Chapter 2 of \cite{co2}), in our case these two coincide.
\end{Rem}

\begin{Tex}
Given a vertex $V \in \ct$ and an element $w \in \bd$, then $w$ can be uniquely represented by a word starting at vertex $V$. This specific representation of $w$ will be denoted by $[V,w)$. \\
Given two distinct points $\xi_{1},\xi_{2} \in \bd$, we can form the unique geodesic connecting these, which will be denoted by $]\xi_{1},\xi_{2}[$.
\end{Tex}

\begin{Defi}[Distance on the boundary]
Let $O$ be a distinguished vertex in the universal covering tree $\ct$, called the origin. The distance $d_{X,O}$ on $\bd$ is defined by
\bea
d_{X,O}: \bd \times \bd & \to & \R_{\geq 0}  \nn \\
d(w,w') & = & 2^{-n},
\eea
with
\bea
n & = & \# \{ [O,w) \cap [O,w') \} \in \N \cup  \{+\infty \},
\eea
the number of coinciding edges of the two words, with the convention $2^{-\infty}=0$.
\end{Defi}

\begin{Tex}
The distance turns $\bd$ into a metric space which induces a topology on the boundary $\bd$.\\
For fixed $v=[O,v)=v_{0}v_{1}\ldots \in \bd$ and $N \in \N_{>0}$, basic open balls around $v$ of radius $2^{-N}$ are described by
\bea
B(v,2^{-N}) & = & \{ w \in \bd \mid d(w,v) < 2^{-N} \}  \nn \\
& = & \{ w \in \bd \mid [O,w) = v_{0}v_{1}\ldots v_{N+1}Q \text{ where $Q$ runs over} \nn \\
& & \text{all infinite paths with no backtracking starting at the} \nn \\
& & \text{endpoint of the path } P=v_{0} \ldots v_{N+1} \}
\eea
We will denote this open ball by $U_{P}$ (see figure \ref{fig:circulator:1}).
\begin{figure}[htbp]
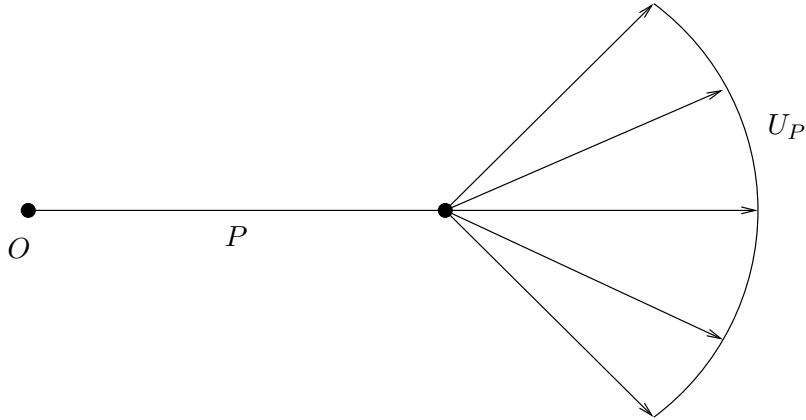

  \begin{center}
   \input open_t 
    \caption{The open set $U_{P}$.}
    \label{fig:circulator:1}
   \end{center}
 \end{figure}
\end{Tex}

\begin{Prop}
The boundary is a totally disconnected Hausdorff space and the set $U_{P}$ as above is clopen (open and closed).
\end{Prop}

\begin{proof}
For the Hausdorff property: Let $[O,w),[O,w')$ be two distinct elements of $\bd$, then there exists an $N$ such that $w_{N} \neq w'_{N}$, define $P = w_{0}\ldots w_{N}$ and $Q=w'_{0}\ldots w'_{N}$, then $w \in U_{P}$ and $w' \in U_{Q}$ and $U_{P} \cap U_{Q} = \emptyset$, so $\bd$ is Hausdorff.\\
Now we show that $U_{P}^{c}$ is open:
\bea
U_{P}^{c} & = & \bigcup_{R \neq P, l(R)=l(P)}U_{R},
\eea
where $l$ assigns to a word its length as in Definition \ref{distance}. This is a union of open sets and hence open.\\
Finally, suppose that $w$ and $w'$ are in a connected component $A$ of $\bd$, then, $U_{P} \cap A$ and $U_{P}^{c} \cap A$ is a separation of $A$ and hence $w=w'$, so $\bd$ is totally disconnected.
\end{proof}

\begin{Defi}[Critical exponent]
Let $\Gamma$ be the fundamental group of $X$, acting as isometry on $\ct$. Consider the formal Poincar\'e series,
\bea
\sum_{\gamma \in \Gamma}e^{-s d(O,\gamma O)}
\eea
This series converges absolutely for sufficiently big enough $\Re(s) \in \R \cup \{ \infty \}$ (with the convention $e^{-\infty}=0$).\\
Define $\delta \in \R_{\geq 0} \cup \{ \infty \}$ as the critical exponent, i.e. for $\Re(s) > \delta$ the series converges and for $\Re(s) < \delta$, the series diverges. If $\delta \notin \{0,\infty\}$, then $\Gamma$ is called a \emph{divergence group} with respect to its action on $X$ (see for instance \cite{her}). In our case $\Gamma$ acts as divergence group, see Remark \ref{check}.
\end{Defi}

\begin{Defi}[Patterson-Sullivan measure on the boundary] \label{patter}
Let $\delta$ be the critical exponent of $\Gamma$ acting on $E_{X}$ and let $D_{P}$ be the Dirac measure at the point $P \in E_{X}$, i.e. $D_{P}(U)=1$ if $P \in U$ and $0$ else.\\
Define the Patterson-Sullivan measure as
\bea
\mu =\mu_{\text{PS},O}=\lim_{s \downarrow \delta} \big( \frac{\sum_{\gamma \in \Gamma}e^{-s d(O,\gamma O)}D_{\gamma O}}{\sum_{\gamma \in \Gamma}e^{-s d(O,\gamma O)}}\big).
\eea
(see \cite{her}). This converges weakly to a probability measure on $X \cup \bd$ with support on the boundary $\bd$ and hence turns $\bd$ into a measure space.
\end{Defi}

\section{The spectral triple}
In this section we will define a finitely summable spectral triple in the sense of Connes, a non-commutative analogue of Riemannian manifolds, and we will study its associated zeta function formalism.

\begin{Defi}[Spectral triples] \label{s3}
A unital \emph{spectral triple} consists of a triple $(A,H,D)$, where $A$ is a unital $C^{*}$-algebra, $H$ a Hilbert-space on which $A$ faithfully acts by bounded operators and $D$ is an unbounded, self adjoint operator, densely defined on $H$, such that $(D-\lambda)^{-1}$ is a compact operator for $\lambda \notin \text{Spec}(D)$, and such that all the commutators $[D,a]$ are bounded operators for $a$ in a dense, involutive subalgebra $A_{\infty} \subset A$ satisfying $a(\text{Dom}(D)) \subset \text{Dom}(D)$ (see \cite{con}).\\
A triple $(A,H,D)$ is called $p$-summable if the trace $\Tr((1+D^{2})^{-p/2})$ is finite.
\end{Defi}

\begin{Defi}
Let $\{ \gamma_{1},\ldots,\gamma_{g}\}$ be a set of generators of the fundamental group $\Gamma$ of $X$. Define $F_{g}$ as the graph consisting of one vertex with $g$ loops attached to it, corresponding to the generators of the fundamental group. Denote by $\bdf$ the boundary of the universal covering tree of this graph  (see figure \ref{fig:circulator:3}). This corresponds to the boundary of the Cayley-graph of $\Gamma$, with respect to the chosen generators of the fundamental group (see for instance \cite{grom}).
\end{Defi}

\begin{figure}[htbp]
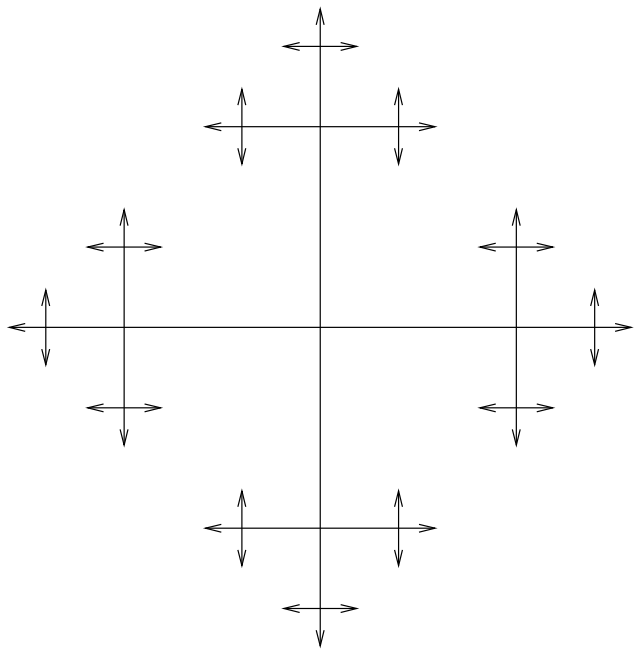

  \begin{center}
   \input f2_t 
    \caption{The covering tree of $F_{2}$.}
    \label{fig:circulator:3}
   \end{center}
\end{figure}

We will use Theorem 4.1 of chapter $4$ in \cite{co2} with $Y=E_{X}$ and $G=\Gamma$, the fundamental group of $X$, whose boundary is denoted by $\partial F_{g}$, which satisfy the hypotheses of the theorem:

\begin{Prop} \label{giso}
Let $Y$ be a proper geodesic space and let $G$ be an isometry group of $Y$, acting properly discontinuous, such that the quotient $Y/G$ is compact. Then $G$ is hyperbolic if and only if $Y$ is. Moreover, if $G$ (and hence $Y$) is hyperbolic, then there is a canonical homeomorphism:
\bea
\Phi_{Y}: \partial G \to \partial Y.
\eea
\end{Prop}

\begin{Tex}
For a definition of hyperbolic I refer to \cite{grom}. Properly discontinuous means that for any compact set $K \subset Y$ the set
\bea
\{ g \in G \mid g(K) \cap K \neq \emptyset \}
\eea
is finite.\\
After fixing an origin $O$, the map $\Phi_{Y}$ is induced by the map
\bea
\phi_{Y}: G & \to & Y \nn \\
g & \mapsto & g(O),
\eea
which is a quasi-isometry, which means that there are $K \geq 1$ and $C \geq 0$ such that $\frac{1}{K}d(g_{1},g_{2}) - C \leq d(\phi_{Y}(g_{1},g_{2})) \leq K d(g_{1},g_{2}) + C$ for all $g_{1},g_{2} \in G$. By Theorem 2.2 of \cite{co2}, the map $\phi_{Y}$ then induces a map on the boundaries. The map $\phi_{Y}$ is $G$-equivariant with respect to the left action of $G$ on itself (see next Definition \ref{equiv}) and hence the map $\Phi_{Y}$ is $G$-equivariant as well.
\end{Tex}

\begin{Defi}[Equivariant maps]\label{equiv}
Let $Y_{1},Y_{2}$ be sets. Suppose that the group $G_{i}$ acts on $Y_{i}$ and let $\alpha:G_{1} \to G_{2}$ be a group homomorphism, then a map $f:Y_{1} \rightarrow Y_{2}$ is called \emph{equivariant} (with respect to $G_{1},G_{2},\alpha$), if for all $g \in G_{1}, y \in Y_{1}$ we have
\[
f(g.y)=\alpha(g).f(y)
\]
\end{Defi}
\begin{Rem}
Note that the map $\Phi_{Y}$ is not determined by the abstract group $G$, but depends on the representation of $G$ as an isometry group of $Y$ and on a choice of generators when passing to the Cayley graph. Different choices do not give isometric, but quasi-isometric Cayley graphs.
\end{Rem}

\begin{Defi}[Algebra of functions, Hilbert space]
Let $A_{X}$ be the algebra of continuous, $\C$-valued functions on $\p F_{g}$, i.e. $C(\p F_{g},\C)$. This algebra contains the subalgebra $A_{X,\infty}=C(\p F_{g},\Z) \otimes_{\Z} \C$ of locally constant functions.
Let $\mu$ be the Patterson-Sullivan measure on $\bd$ (see Definition \ref{patter}), then the measure $\Phi_{X}^{*}(\mu)$ (see Proposition \ref{giso}) induces an inner product by integration on $\p F_{g}$ and hence a norm on $A_{X}$. Let $H_{X}$ be the completion of $A_{X}$ with respect to this norm.
\end{Defi}

\begin{Rem}
Note that after fixing the genus, only $H_{X}$ depends on the graph $X$, especially the innerproduct on $H_{X}$.\\
We will use the notation $A=A_{X}$ and $A_{\infty}=A_{X,\infty}$ and $H=H_{X}$ for the algebras.
\end{Rem}

\begin{Prop}
The subalgebra $A_{\infty}$ is dense in $A$.
\end{Prop}

\begin{proof}
The boundary is a totally disconnected Hausdorff space. The sets $U_{P}$ form a basis for the topology and hence the result. 
\end{proof}

\begin{Defi}[Dirac operator] \label{dirac}
The space $A_{\infty}$ has a natural filtration $\{A_{n}\}_{n \geq 0}$ by setting
\bea
A_{n}=\text{span}_{\C} \{ 1_{P} \mid P \text{ a finite word of length } \leq n \}.
\eea
This filtration is inherited by its completion $H$. Let $P_{n}$ be the orthogonal projection operator on $A_{n}$, with respect to the inner product on $H$, 
\bea
P_{n}: H \to A_{n} \subset H.
\eea
The Dirac operator,
\bea
D & : & H \to H,
\eea
is defined by 
\bea
D & = &\lambda_{0} P_{0}+ \sum_{n \geq 1} \lambda_{n}Q_{n},
\eea
with
\bea
Q_{n} & = & P_{n} - P_{n-1}, \\
\lambda_{n} & = & (\dim A_{n})^{3}.
\eea
\end{Defi}

\begin{Rem}
As it turns out, we will only need the operator $|D|$ (via the zeta functions), which leaves room for a possibly interesting sign. For instance, in the case of non-commutative manifolds, the sign $F$ in $D=F |D|$, gives rise to the fundamental $K$-cohomology class of the manifold and plays an important role in the index formula (see for instance \cite{con3,nig}). A sign $F$ must satisfy $F^{2}=1$, $[F,|D|]=0$, must be bounded and $[a,F]$ must be compact.
\end{Rem}

\begin{Lemm}
We have:
\renewcommand{\labelenumi}{\alph{enumi}.}
\begin{enumerate}
\item $\dim A_{0} = 1$,
\item $\dim A_{1} \ominus A_{0} = 2g-1$
\end{enumerate}
For $n \geq 1$ we have
\begin{enumerate}
\renewcommand{\labelenumi}{\roman{enumi}.}
\item $\dim A_{n} = 2g(2g-1)^{n-1}$,
\item $\dim A_{n+1} \ominus A_{n} = (2g)(2g-2)(2g-1)^{n-2}$.
\end{enumerate}
\end{Lemm}

\begin{proof}
From the origin $2g$ edges emerge and every next step there are $2g-1$ choices, giving the results.
\end{proof}

\begin{Prop}
The data $(A,H,D)$ as constructed forms a $1$-summable spectral triple.
\end{Prop}

\begin{proof}
The $*$-operation is complex conjugation, $A$ acts on $H$ by multiplication which is bounded. The operator $D$ is real and hence self-adjoint and the compactness of $(D+\lambda)^{-1}$ for $\lambda \notin \text{Spec}(D)$ is easily checked. For $a \in A_{n}$ and $m>n$ we have $P_{m}(a)=P_{m-1}(a)$. So $P_{m}|_{A_{k}}=P_{k}$ for $k \leq m$ and so $[Q_{m},a]=0$ for $m > n$, in particular $[D,a]$ is a finite linear combination of finite rank operators of the form $[Q_{i},a]$ and hence bounded.\\
For the $1$-summability, note that
\bea
\Tr((1+D^{2})^{-1/2}) & = & \sum_{n=0}^{\infty}(1+\lambda_{n}^{2})^{-1/2}(\dim A_{n} -\dim A_{n-1}) \nn \\
& \leq & \sum_{n=0}^{\infty}(1+\lambda_{n}^{2})^{-1/2} \dim A_{n}   \nn\\
& < & \sum_{n=0}^{\infty}(\dim A_{n})^{-2}  \nn \\
& \leq &\sum_{n=0}^{\infty} (n+1)^{-2}  \nn\\
& < & \infty,
\eea
where we used that $\dim A_{n} \geq n+1$.
\end{proof}

\section{Zeta functions}

\begin{Defi}[Zeta functions]
Finitely summable spectral triples give rise to zeta functions. For any $a \in A_{\infty}$ and $\Re(s) \ll 0$, the zeta function is defined by
\bea
\zeta^{X}_{a}(s) & = & \Tr_{H}(a|D|^{s}).
\eea
From the general framework of finitely summable spectral triples, it follows that this function can be extended to a meromorphic function on $\C$ (see Proposition $1$ of \cite{con2}).
\end{Defi}

\begin{Tex}
Let us expand $\zeta_{a}^{X}$ for later convenience. Let $I_{m}$ stand for an inductively obtained orthogonal basis for $A_{m}$ with $A_{k} \subset A_{k+1}$ for all $k \in \N_{\geq 0}$ (obtained by for instance a Gram-Schmidt orthogonalization procedure). By convention, $I_{0}$ corresponds to the constant functions and $I_{-1}=\emptyset$.\\
\bea
\zeta^{X}_{a}(s) & = & \Tr_{H}(a|D|^{s}) \nn \\
& = & \sum_{n \geq 0} \sum_{\Psi \in I_{n}-I_{n-1}} \langle \Psi| \lambda_{0}^{s}a P_{0}+\sum_{m \geq 1}\lambda_{m}^{s}a Q_{m} |\Psi \rangle  \nn\\
& = & \sum_{n \geq 0} \lambda_{n}^{s} c_{n}(a),
\eea
with
\bea
c_{n}(a) & := & \sum_{\Psi \in I_{n}-I_{n-1}} \langle \Psi|a|\Psi \rangle.
\eea
\end{Tex}

\begin{Defi}[Equivalence of zeta functions] \label{equivzeta}
As noted, the zeta functions constructed above depend on a choice of origin and on the representation of $\Gamma$ as group of isometries of $E_{X}$. To show the dependence in the notation we write $\zeta^{X,O,\alpha}_{a}$. Here $O \in X$ denotes the (arbitrary chosen) origin and $\alpha$ denotes a representation of $\Gamma$ as isometry group of $X$, including a minimal set of generators. Denote by $\text{R}(\Gamma,X)$ the set of all such $\alpha$'s.
Let
\bea
\ZX=\{(\zeta^{X,O,\alpha}_{a})_{a \in A_{\infty}}\ \mid O \in X, \alpha \in \text{R}(\Gamma,X)\}
\eea
denote the set of all indexed rows of zeta functions, which can be obtained by varying the origin and $\alpha$'s as described.
\end{Defi}

\begin{Prop} \label{zetaone}
The zeta function $\zeta_{1}^{X}$ does not depend on any choices and is equivalent to knowing the genus of $X$.
\end{Prop}

\begin{proof}
We will explicitly calculate it. Let $\Re(s) \ll 0$, then\\
\bea
\Tr(|D|^{s}) & = & 1 + \sum_{n \geq 1} \lambda_{n}^{s}(\dim A_{n} -\dim A_{n-1}) \nn \\
& = & 1 + (2g)^{3s}(2g-1) +  \nn\\
& & \sum_{n \geq 2} ((2g)(2g-1)^{n-1})^{3s}\cdot (2g)(2g-1)^{n-2}(2g-2)  \nn\\
& = & 1 + (2g)^{3s}(2g-1)\Bigl\{ \frac{1-(2g-1)^{3s-1}}{1-(2g-1)^{3s+1}} \Bigr\} \label{zeta1}.
\eea
The first order expansion around $s=-\infty$ is
\bea
1+(2g)^{3s}(2g-1).
\eea
So the formula determines $g$ and on the other hand $g$ determines $\zeta_{1}$ by the formula above.
\end{proof}

\begin{Prop} \label{meas}
Let $X_{1},X_{2}$ be graphs with the same genus $g \geq 2$. If $\zeta_{a}^{X_{1}}=\zeta_{a}^{X_{2}}$ for all $a \in A_{\infty}$, then the induced measures $\Phi^{*}_{X_{i}}(\mu_{i})$ on $\partial F_{g}$ are equal (with these specific choices of origin and representation). \\
More precise, $\lim_{s \to -\infty} \zeta_{a}^{X_{i}}(s)$ equals $\int_{\partial F_{g}} a \,  \ud (\Phi^{*}_{X_{i}}(\mu_{i}))$.
\end{Prop}

\begin{proof}
It suffices to show that the zeroth term in the expansion of $\zeta_{a}^{X}$ is given as in the theorem as this equals $\lim_{s \to -\infty} \zeta_{a}^{X}(s)$.\\
For this it suffices to show the equality for a basis, i.e. it suffices to show that $\nu(U_{P})=\int_{\partial F_{g}} 1_{P} d\nu$ for all $P$, here $\nu$ is one of the induced measures and $P$ is a finite path starting in $O$ as before. \\
Fix $P$ and take the canonical orthogonal basis $B_{P}$ for $A_{|P|}$, i.e.
\bea
B_{P}=\{ \frac{1_{Q}}{||1_{Q}||} \mid \text{with }l(Q)=l(P)\}.
\eea
The constant term of the zeta function corresponds to $\Tr(a P_{0})$.
First we will trace only over $A_{|P|}$ and get the result. Then we prove that a refinement of the basis does not change the trace, proving the theorem.\\ 
We have
\bea
\Tr_{A_{|P|}}(a P_{0}) & = & \sum_{w \in B_{P}} \langle w | 1_{P} P_{0} | w \rangle  \nn\\
& = & \frac{\int_{\partial F_{g}} 1_{P} P_{0}(1_{P})\, \ud \nu}{\int_{\partial F_{g}} 1_{P} 1_{P} \, \ud \nu}  \nn\\
& = & P_{0}(1_{P})\frac{\nu(U_{P})}{\nu(U_{P})}  \nn\\
& = & P_{0}(1_{P}),
\eea
here we used that $P_{0}(1_{P})$ is a constant function and by abuse of notation we denoted its value by $P_{0}(1_{P}) \in \mathbb{C}$ as well.
The projection is characterized by
\bea
1_{P}-P_{0}(1_{P}) \perp A_{0}
\eea
and because $A_{0} \cong \mathbb{C}$, this is equivalent with
\bea
\int_{\partial F_{g}} (1_{P}-P_{0}(1_{P})) \, \ud \nu=0
\eea
and so $P_{0}(1_{P})=\nu(U_{P})$, proving the first assertion. \\
Now let $B' \supset B$ be an orthogonal extension of the basis $B$, in particular we have for $v \in B'\setminus B$, $\int 1_{P} \cdot v \ud \mu= 0$, so
\bea
\Tr_{B'}(1_{P}P_{0}) & = & \Tr_{B}(1_{P}P_{0}) + \Tr_{B' \setminus B}(1_{P}P_{0})  \nn \\
& = & \mu(U_{P})+\sum_{v \in B'} \int_{\partial F_{g}} v \cdot 1_{P} \cdot P_{0}(v) \, \ud \nu  \nn \\
& = & \mu(U_{P})+\sum_{v \in B'} P_{0}(v)\int_{\partial F_{g}} v \cdot 1_{P} \, \ud \nu  \nn \\
& = & \mu(U_{P})+\sum_{v \in B'} P_{0}(v) \cdot 0   \nn \\
& = & \mu(U_{P}),
\eea
proving the proposition.
\end{proof}

\begin{Rem}
The computation of the corresponding term in \cite{cor} is wrong and should be replaced by a calculation similar to the above. This only affects the proofs, not the results, of \cite{cor}.
\end{Rem}

\section{The main theorem}
\begin{Theo}[Main theorem] \label{czet}
Let $X_{1},X_{2}$ be finite, connected graphs of genus $g \geq 2$ and valencies $\geq 3$.
Then either
\bea
\ZXo \cap \ZXt = \emptyset,
\eea
or
\bea
\ZXo = \ZXt \text{ and } X_{1} \cong X_{2} \text{ as graphs}.
\eea
Here the intersection is defined as follows:\\
One starts by comparing at the unit of the algebras, i.e. $\zeta_{1}^{X_{i}}$: If these are different, the intersection is defined to be $\emptyset$; If they are the same, the genus $g$ is the same by Proposition \ref{zetaone}. Now consider $\partial F_{g}$ and fix the algebra $A_{\infty}$. The infinitely long rows $(\zeta_{a}^{X_{i}})_{a \in A_{\infty}} \in \zeta[X_{i}]$ are now indexed by the same algebra and we can compare elements $(\zeta_{a}^{X_{i}})_{a \in A_{\infty}} \in \zeta[X_{i}]$.
\end{Theo}

Before proving this theorem we will need some definitions and results.
\begin{Defi}[Cross-ratio, M\"obius]
For $\xi_{1},\ldots,\xi_{4}$ four distinct points on the boundary $\bd$, define the cross-ratio as:
\bea
b(\xi_{1},\xi_{2},\xi_{3},\xi_{4}):=\frac{d(\xi_{3},\xi_{1})}{d(\xi_{3},\xi_{2})}:\frac{d(\xi_{4},\xi_{1})}{d(\xi_{4},\xi_{2})}
\eea
This can be written as $\exp(L)$ with $L$ (up to sign) the distance between the geodesics $]\xi_{1},\xi_{3}[$ and $]\xi_{2},\xi_{4}[$, see figure \ref{fig:circulator:2}.\\
A function $f$ preserving the cross-ratio, i.e. such that for all distinct points $\xi_{1},\ldots,\xi_{4}$ we have
\[
b(\xi_{1},\xi_{2},\xi_{3},\xi_{4})=b(f(\xi_{1}),f(\xi_{2}),f(\xi_{3}),f(\xi_{4})),
\]
is called \emph{M\"obius} (see \cite{her},\cite{co1}).
\end{Defi}

\begin{figure}[htbp]
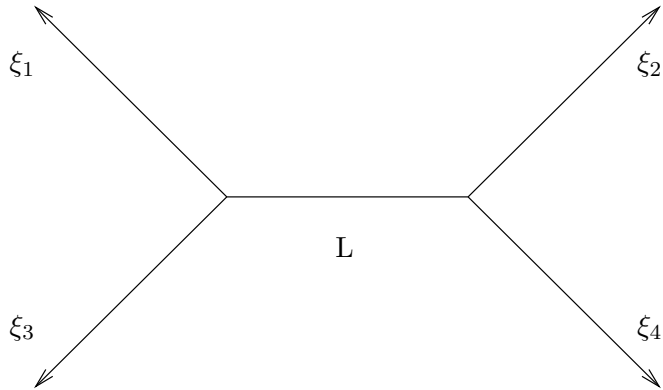

  \begin{center}
   \input cross_t 
    \caption{The cross-ratio.}
    \label{fig:circulator:2}
   \end{center}
\end{figure}

\begin{Lemm}
Let $\tilde{\phi}: \p X_{1} \to \p X_{2}$ be an (equivariant) isomorphism which is M\"obius, then $\tilde{\phi}$ induces an (equivariant) isomorphism of trees $F:E_{X_{1}} \to E_{X_{2}}$ and hence an isomorphism of graphs $X_{1} \cong X_{2}$.
\end{Lemm}

\begin{proof}
A similar construction is used in \cite{co1}.\\
We define a map $F$ as follows:\\
Let $x \in V(E_{X_{1}})$, a vertex in the covering graph. Pick $\xi_{1},\xi_{2},\xi_{3} \in \p X_{1}$ such that $x$ is the center of the tripod induced by these (see figure \ref{fig:tripod}). Define $F(x) \in V(E_{X_{2}})$ as the unique center of the tripod of $\tilde{\phi}(\xi_{1}),\tilde{\phi}(\xi_{2}),\tilde{\phi}(\xi_{3}) \in \p X_{2}$. To show that this is well-defined, if $x$ is also the center of $\xi_{1},\xi_{2},\xi_{4}$, then this is the same as saying that the geodesics $]\xi_{1},\xi_{2}[$ and $]\xi_{3},\xi_{4}[$ are touching, i.e. $L=0$, this is preserved because $\tilde{\phi}$ is M\"obius and hence the map is well-defined. The map is an isomorphism on the vertex sets, because the construction can be reversed.\\
To see that this extends to a map of covering trees, let $e$ be an edge of length $L_{1}$, then the initial and terminal vertex of $e$ are mapped to two distinct vertices in $V(E_{X_{2}})$. Connect these by a path $P$ of length $L=L_{1}$ (because the map is M\"obius it is $L_{1}$ again). We have to show that $P$ is an edge again. If not there is a new vertex $e'$ on $P$, this point is mapped by the inverse to a point on distance on distance $K,K'<L$ from the initial and terminal vertex of $e$ which is impossible, hence $F$ is an isomorphism of graphs.\\
By well-definedness, if $\tilde{\phi}$ is equivariant for some group-action, then so is $F$.
\end{proof}

\begin{figure}[htbp]
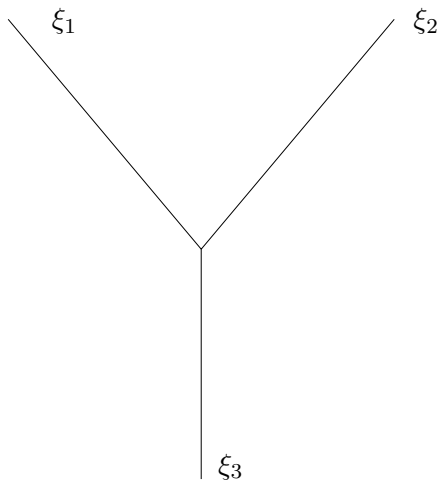

  \begin{center}
    \input tripod_t 
    \caption{The tripod.}
    \label{fig:tripod}
   \end{center}
 \end{figure}

In the proof of the main theorem we will use the following theorem (see \cite{her}, Theorem A):\\
\begin{Theo}
Let $Y_{1},Y_{2}$ be locally compact complete $\text{CAT}(-1)$ metric spaces.
Let $G_{1}$ and $G_{2}$ be discrete groups of isometries of $Y_{1}$ and $Y_{2}$, having the same critical exponent.
Suppose that $G_{2}$ is a divergence group. Let $\tilde{\phi}: \p Y_{1} \to \p Y_{2}$ be a Borel map, equivariant for some morphism $G_{1} \to G_{2}$, which is non-singular with respect to the Paterson-Sullivan measures.
Then $\tilde{\phi}$ is M\"obius on the limit set of $G_{1}$.
\end{Theo}

\begin{Tex} \label{check}
Let us apply this theorem for $Y_{1}=E_{X_{1}}$, $Y_{2}=E_{X_{2}}$, the covering trees of the graphs $X_{1}$ and $X_{2}$ respectively, having the same genus $g \geq 2$ with $G_{1}= \Gamma = G_{2}$, the fundamental groups of $X_{1}$ and $X_{2}$.\\
The trees under consideration are locally compact and are $\text{CAT}(k)$ metric spaces for any $k$ (see for instance \cite{brid}) and the fundamental group acts as isometries.\\
To construct an equivariant map $\tilde{\phi}$, note that by Proposition \ref{giso} we have an equivariant homeomorphism $\Phi_{X_{i}}: \p F_{g} \to \p X_{i}$, so $\tilde{\phi} = \Phi_{X_{2}} \circ \Phi_{X_{i}}^{-1}: \partial X_{1} \to \partial X_{2}$ is an equivariant homeomorphism.\\
The group $\Gamma$ acts as divergence group on $E_{X_{i}}$. This is because the Poincar\'e series is bounded by the one of the covering tree of $F_{g}$ (because the group acts free and $F_{g}$ is a rectract of $X$) and the graph $F_{g}$ has $\delta = \frac{log(2g-1)}{l}$ as critical exponent, as one easily computes.  Furthermore, $\delta>0$ because the fundamental group is infinite. So we have proven that $\Gamma_{i}$ acts as divergence group on $X_{i}$.\\
Note that if we scale the metric of the covering tree by $\lambda$, the critical exponent scales by $\lambda^{-1}$, so by rescaling we can assume that the critical exponents are the same. Combining this with the previous lemma, we get the following proposition:
\end{Tex}

\begin{Prop} \label{rigid}
Let $X_{1}, X_{2}$ be graphs of genus $g \geq 2$ with covering trees $E_{X_{1}}, E_{X_{2}}$. If $\tilde{\phi} = \Phi_{X_{2}} \circ \Phi_{X_{i}}^{-1}: \p X_{1} \to \p X_{2}$ is non-singular with respect to the Patterson-Sullivan measures, then $\tilde{\phi}$ induces an equivariant isomorphism of $E_{X_{1}} \to E_{X_{2}}$ and hence an isomorphism of $X_{1} \to X_{2}$.
\end{Prop}

\begin{proof}[Proof of the main theorem, Theorem \ref{czet}]
First, suppose that the intersection defined as above is empty. This implies that the graphs are not isomorphic, because isomorphic graphs give rise to the same set $\zeta[X]$.\\
Suppose now that $\ZXo \cap \ZXt \neq \emptyset$. We must show that the graphs are isomorphic.\\
The intersection is nonempty, so by definition they have the same genus $g$. Suppose that $(\zeta^{X_{i}}_{a})_{a \in A_{\infty}}$ is in the intersection.\\
We have the following commuting diagram
\begin{displaymath}
\begin{diagram}
\node{(F_{g},\Phi_{X_{1}}^{*}(\mu_{1}))} \arrow{e,t}{\text{id}} \node{(F_{g},\Phi_{X_{2}}^{*}(\mu_{2}))}\\
\node{(\p X_{1},\mu_{1})} \arrow{e,t}{\tilde{\phi}} \arrow{n,t}{\Phi_{X_{1}}} \node{(\p X_{2},\mu_{2}).} \arrow{n,t}{\Phi_{X_{2}}}
\end{diagram}
\end{displaymath}
The zeta functions are equal, so we know that $\Phi_{X_{1}}^{*}(\mu_{1})=\Phi_{X_{2}}^{*}(\mu_{2})$ and hence $\mu_{2}=\tilde{\phi}^{*}(\mu_{1})$ and so $\tilde{\phi}$ is non-singular with respect to the Patterson-Sullivan measures and hence the previous proposition applies to the equivariant map $\tilde{\phi}$.
\end{proof}

\begin{Theo} \label{ctrip}
The spectral triple $\mathcal{S}_{X}$ determines the graph $X$.
\end{Theo}

\begin{proof}
The spectral triple determines the zeta functions which by the main theorem determines the graph.
\end{proof} 

\begin{Rem}
It would be desirable to give a functorial version of the construction. The graphs under consideration form a category in the obvious way, the objects being graphs and the arrows being graph-homomorphisms.
By Theorem \ref{czet} the map
\bea
(X,O,\alpha) \mapsto \left(\zeta^{X,O,\alpha}_{a}\right)_{a \in A_{\infty}}
\eea
descends to a bijection of isomorphism classes
\bea
[X] \mapsto \zeta[X],
\eea
where the isomorphism classes and the image on the right are defined as the images of the isomorphism classes of $X$. Furthermore, by Theorem \ref{ctrip} this induces an equivalence relation on spectral triples coming from the construction as well, so this defines a map:
\bea
\{ \text{graphs} \} / \{ \text{isomorphisms} \} & \to & \{ \text{spectral triples} \} / \{ \text{induced isomorphisms} \} \nn \\
\left[X\right] & \mapsto & \left[\mathcal{S}_{X}\right]
\eea
At this very moment however, there does not exist a (generally accepted) category of spectral triples (see \cite{cats3} for some ideas using strict equivalence, but not (yet) involving Morita equivalence). The spectral triples constructed here are commutative, which is restrictive, and therefore missing morphisms which are not clearly visible in the commutative setting. For instance, two commutative algebras are Morita equivalent if and only if these are isomorphic as algebras, a statement which is not true in the non-commutative setting (see for instance \cite{ra}).
\end{Rem}
\bibliography{database}
\bibliographystyle{amsplain}
\end{document}